\documentclass[11pt]{amsart}
\usepackage{epsfig}
\usepackage{graphicx}
\usepackage{amscd}
\usepackage{amsmath}
\usepackage{amsxtra}
\usepackage{amsfonts}
\usepackage{amssymb}

\newtheorem{theorem}{Theorem}[section]
\newtheorem{corollary}[theorem]{Corollary}
\newtheorem{lemma}[theorem]{Lemma}
\newtheorem{proposition}[theorem]{Proposition}

\theoremstyle{definition}
\newtheorem{definition}[theorem]{Definition}
\newtheorem{remark}[theorem]{Remark}

\newtheorem{example}[theorem]{Example}
\theoremstyle{remark}

\renewcommand{\theclaim}{\textup{\theclaim}}

\newtheorem*{acknowledgements}{Acknowledgements}

\numberwithin{equation}{section}

\def\openone

{\mathchoice

{\hbox{\upshape \small1\kern-3.3pt\normalsize1}}

{\hbox{\upshape \small1\kern-3.3pt\normalsize1}}

{\hbox{\upshape \tiny1\kern-2.3pt\SMALL1}}

{\hbox{\upshape \Tiny1\kern-2pt\tiny1}}}

\makeatletter

\newbox\ipbox

\newcommand{\ip}[2]{\left\langle #1\, , \,#2\right\rangle}
\newcommand{\diracb}[1]{\left\langle #1\mathrel{\mathchoice

{\setbox\ipbox=\hbox{$\displaystyle \left\langle\mathstrut
#1\right.$}

\vrule height\ht\ipbox width0.25pt depth\dp\ipbox}

{\setbox\ipbox=\hbox{$\textstyle \left\langle\mathstrut
#1\right.$}

\vrule height\ht\ipbox width0.25pt depth\dp\ipbox}

{\setbox\ipbox=\hbox{$\scriptstyle \left\langle\mathstrut
#1\right.$}

\vrule height\ht\ipbox width0.25pt depth\dp\ipbox}

{\setbox\ipbox=\hbox{$\scriptscriptstyle \left\langle\mathstrut
#1\right.$}

\vrule height\ht\ipbox width0.25pt depth\dp\ipbox}

}\right. }

\newcommand{\dirack}[1]{\left. \mathrel{\mathchoice

{\setbox\ipbox=\hbox{$\displaystyle \left.\mathstrut
#1\right\rangle$}

\vrule height\ht\ipbox width0.25pt depth\dp\ipbox}

{\setbox\ipbox=\hbox{$\textstyle \left.\mathstrut
#1\right\rangle$}

\vrule height\ht\ipbox width0.25pt depth\dp\ipbox}

{\setbox\ipbox=\hbox{$\scriptstyle \left.\mathstrut
#1\right\rangle$}

\vrule height\ht\ipbox width0.25pt depth\dp\ipbox}

{\setbox\ipbox=\hbox{$\scriptscriptstyle \left.\mathstrut
#1\right\rangle$}

\vrule height\ht\ipbox width0.25pt depth\dp\ipbox}

} #1\right\rangle}

\newcommand{\cj}[1]{\overline{#1}}

\newcommand{\bz}{\mathbb{Z}}

\newcommand{\br}{\mathbb{R}}
\newcommand{\bc}{\mathbb{C}}

\def\blfootnote{\xdef\@thefnmark{}\@footnotetext}

\renewcommand{\mod}{\operatorname{mod}}

\input xy
\xyoption{all}
\usepackage{amssymb}

\newcommand{\supp}[1]{\text{supp} (#1)}

\def\H{\mathcal{H}}

\def\-{^{-1}}

\begin{document}

\title[Fractal measures with affine scales]{Fourier duality for fractal measures with affine scales}
\author{Dorin Ervin Dutkay}
\blfootnote{Supported in part by the National Science Foundation.}
\address{[Dorin Ervin Dutkay] University of Central Florida\\
	Department of Mathematics\\
	4000 Central Florida Blvd.\\
	P.O. Box 161364\\
	Orlando, FL 32816-1364\\
U.S.A.\\} \email{ddutkay@mail.ucf.edu}

\author{Palle E.T. Jorgensen}
\address{[Palle E.T. Jorgensen]University of Iowa\\
Department of Mathematics\\
14 MacLean Hall\\
Iowa City, IA 52242-1419\\}\email{jorgen@math.uiowa.edu}

\thanks{} 

\begin{abstract}
   For a family of fractal measures, we find an explicit Fourier duality. The measures in the pair have compact support in $\br^d$, and they both have the same matrix scaling. But the two use different translation vectors, one by a subset $B$ in $\br^d$, and the other by a related subset $L$. Among other things, we show that there is then a pair of infinite discrete sets $\Gamma(L)$ and $\Gamma(B)$ in $\br^d$ such that the $\Gamma(L)$-Fourier exponentials are orthogonal in $L^2(\mu_B)$, and the $\Gamma(B)$-Fourier exponentials are orthogonal in $L^2(\mu_L)$. These sets of orthogonal ``frequencies'' are typically lacunary, and they will be obtained by scaling in the large. The nature of our duality is explored below both in higher dimensions and for examples on the real line.

    Our duality pairs do not always yield orthonormal Fourier bases in the respective $L^2(\mu)$-Hilbert spaces, but depending on the geometry of certain finite orbits, we show that they do in some cases. We further show that there are new and surprising scaling symmetries of relevance for the ergodic theory of these affine fractal measures.  
\end{abstract}
\maketitle \tableofcontents
\section{Introduction}\label{intr}

 Fractal scaling and self-similarity occurs in nature, and in
applications, such as to large communication networks. To understand the
nature of fractal scaling, it has proved useful to develop specific model
cases. Here we focus on such a class, those given by a finite family of
affine transformations.
Iteration in the small yields fractal measures $\mu$, and their support sets
fall into one of the classes of compact fractals embedded in $\br^d$. Iteration
of scale in the large, in turn, leads to a kind of  fractal networks
comprising of a lattice skeleton and lacunarity, degrees of sparsity. We
study when these scales in the large lead to orthogonal families in
$L^2(\mu)$.

 There have been a number of recent papers dealing with and treating a variety of features of the harmonic analysis of fractal measures $\mu$ with affine scales, see for example \cite{DJ06, DJ07a, DJ07b, DJ07c, DJ07d, DJ08, DJ09, GIL09, HL08, IW08,St99a, St99b, St00}.

     In this paper we study pairs of orthogonal complex exponentials in two Hilbert spaces $L^2(\mu)$ for a pair of fractal measures $\mu$ with affine scales. Each $\mu$ in the pair has the same scaling matrix, but the affine mappings making up the iterated function systems (IFSs) are different.

    The particular pairs we study are selected from a certain axiom involving a fixed complex Hadamard matrix. It is indexed by a triple, a matrix and two sets of vectors: more specifically, by an expansive $d$ by $d$ matrix $R$ with integer entries, and by two subsets $B$ and $L$ in $\br^d$ having the same cardinality. We show that when the data $(R, B, L)$ are fixed, we then get a naturally defined pair of affine measures $\mu_B$ and $\mu_L$, each with $R$-selfsimilarity. The two measures arise by taking scaling in the small with powers of the inverse $R^{-1}$, and initiating with the given sets $B$ and $L$. The support of the measures is typically a Cantor fractal, e.g., a Cantor set on the line, or for example a planar Sierpinski set in $\br^2$.  

    We offer a detailed Fourier duality for the pair, and we show that there is a pair of discrete sets $\Gamma(L)$ and $\Gamma(B)$ in $\br^d$ such that the $\Gamma(L)$-Fourier exponentials are orthogonal in $L^2(\mu_B)$, and the $\Gamma(B)$-Fourier exponentials are orthogonal in $L^2(\mu_L)$. These sets of orthogonal ``frequencies'' will be obtained by scaling in the large. The nature of this duality is explored below both in higher dimensions and for examples on the real line.

    Our duality pairs do not always yield orthonormal Fourier bases in the respective $L^2(\mu)$-Hilbert spaces, but we show that they do in some cases. We further show that there are new and surprising scaling symmetries of relevance for the ergodic theory of these affine fractal measures.

 The paper is organized as follows. The problem we consider here began with the question raised first by Fuglede \cite{Fu74} for open subsets $\Omega$ in $\br^d$ concerning orthogonal Fourier bases for the corresponding Hilbert space $L^2(\Omega)$ with respect to Lebesgue measure, i.e., of orthogonal bases (ONBs) of complex exponentials. While the possibility of such Fourier bases is related to tiling properties for $\Omega$, it is known not to be equivalent \cite{Tao04}. For a brief summary of these questions, see for example \cite{JP99,IKP99, IKT03, JoPe91, Jor82a, Jor82b}.

     However, the ideas from \cite{JP99} and \cite{DJ07a} still suggested intriguing connections between spectra and geometry, more specifically that Fourier bases should be tied to an intrinsic selfsimilarity, and further that this can be made precise with the use of Hadamard matrices (see Definition \ref{def2.5}). But in addition, this suggested algorithmic iterations of the similarity transformations, so the introduction of iterated function systems (IFSs) built from a finite family of affine transformations in $\br^d$ (see Definition \ref{def2.1}). Thus, the measures $\mu$ arising in the limit can be expected to come from a dual affine iteration. Further \cite{JP99} suggested a definite relationship between the two sides in such a duality.

      In section \ref{comp} we introduce the Hadamard matrices, and we prove the stated relations for the two IFS-measures in duality. In Proposition \ref{pr2.5} we show that each of the Hilbert spaces $L^2(\mu)$ for the dual system of measures have a natural pair of infinite families of orthogonal Fourier frequencies. The fractal measures arise by iteration in the small, while the orthogonal families by iteration in the large.

    The interplay between the two sides in the duality is made precise in Corollaries \ref{cor2.7} and \ref{cor2.9}. In Theorem \ref{th2.10} we introduce a class of invertible matrices which define equivalence between the two measures. Further we offer examples in section \ref{exam} to the effect that the two Hilbert spaces $L^2(\mu)$ for dual measures may have different Fourier properties.

     Further we show in Lemma \ref{lem3.3} that the nature of the orthogonal families from Proposition \ref{pr2.5} is determined by geometry. By this we mean that for infinite families of orthogonal Fourier frequencies, the obstruction to the ONB property is determined by whether a certain pair of dynamical systems have non-trivial orbits. A further aspect of the duality is pointed out in Lemma \ref{lem5.1} with the use of a pair of transfer operators.

     We offer computable examples in section \ref{exam} and \ref{canto} which we believe may be of independent interest. The example in section \ref{canto} has already been the subject of extensive work, see for example \cite{JP98, DJ09,DHS09}.

  In section \ref{canto} we prove that $L^2(\mu)$ may have orthogonal Fourier bases ONBs (i.e., ONBs of complex exponential frequencies) having arbitrarily small upper Beurling density. Compared with what is usually expected in standard sample theory, this result reveals new features of the harmonic analysis of the Hilbert spaces $L^2(\mu)$.
 
\section{Pairs of spectral fractal measures}\label{comp}

  There is a general construction of measures with intrinsic selfsimilarity due to Hutchinson \cite{Hut81}, but it is closely related to families of infinite product measures considered earlier. The starting point for this construction is a finite family $F$ of contractive transformations in a complete metric space and an assignment of probabilities on $F$. Repeated iteration of the mappings in $F$, and taking averages, then leads to a unique Borel probability measure $\mu$ in the limit; see \eqref{eqi2} below. A special case of this is the measure resulting from a Cantor iteration, and recursive rescaling. Motivated by this, we will be concerned here with the special case when the mappings in $F$ are affine transformations in $\br^d$ for some $d$, see \eqref{eqi4}. In that case, there is a family of complex exponentials indexed by points in $\br^d$, and we will be interested in iterative algorithms for the construction of maximal orthogonal families in $L^2(\mu)$. Ideally we ask for these families to form orthonormal bases (ONBs) for $L^2(\mu)$, i.e., Fourier bases.

\begin{definition}\label{def2.1}
Let $(Y,d)$ be a complete metric space, and let $(\tau_i)_{i=1}^N$ be a finite system of contractive mappings, i.e., there is a constant $c$, $0<c<1$ such that 
\begin{equation}
d(\tau_i(x),\tau_i(y))\leq c d(x,y),\quad(x,y\in Y, i=1,\dots, N)
\label{eqi1}
\end{equation}
Let $(p_i)_{i=1}^N$, $p_i\geq 0$ satisfy $\sum_{i=1}^Np_i=1$. Then by \cite{Hut81} there is a unique probability Borel measure $\mu$ such that 
\begin{equation}
\mu=\sum_{i=1}^Np_i\mu\circ\tau_i^{-1}
\label{eqi2}
\end{equation}
where $\mu\circ\tau^{-1}$ is the measure given by $(\mu\circ\tau^{-1})(E)=\mu(\tau^{-1}(E))$, $E\in\mathfrak B(Y)$=Borel sets, and $\tau^{-1}(E)=\{x\in Y : \tau(x)\in E\}$.
\end{definition}

Here we will study the case $Y=\br^d$ with its usual metric, and we will be interested in families $(\tau_i)$ consisting of transformations of the form 
\begin{equation}
x\mapsto R^{-1}(x+b)
\label{eqi3}
\end{equation}
where $R$ is a fixed expansive matrix (i.e., all eigenvalues $\lambda$ satisfy $|\lambda|>1$) and where $b$ is in a finite subset $B$ of $\br^d$.

We will further restrict to the case of equal probabilities $p_i=1/N$ in \eqref{eqi2}.

The mappings in \eqref{eqi3} will be denoted 
\begin{equation}
\tau_b(x)=R^{-1}(x+b)
\label{eqi4}
\end{equation}

It is known that the corresponding measure $\mu=\mu_{R,B}$ solving \eqref{eqi2} has as support a Cantor set $\br^d$. However, depending on $B$, the measure $\mu$ will vary. It is convenient to chose $B$ such that $0\in B$. 
In that case 
\begin{equation}
\supp{\mu}=\left\{\sum_{k=1}^\infty R^{-k}b_k : b_k\in B\right\}=:X(B)
\label{eqi5}
\end{equation}

Setting 
$$e_t(x):=e^{2\pi i t\cdot x},\quad (t,x\in\br^d)$$
we get 
\begin{definition}
Let $\Gamma\subset\br^d$ be some discrete subset, and let 
$$E(\Gamma):=\left\{e_\gamma : \gamma\in\Gamma\right\}.$$
We say that $\Gamma$ is orthogonal in $L^2(\mu)$ iff the functions in $E(\Gamma)$ are orthogonal, i.e., 
$$\ip{e_{\gamma_1}}{e_{\gamma_2}}=\widehat\mu(\gamma_2-\gamma_1)=0,\mbox{ for all }\gamma_1\neq\gamma_2\in \Gamma,$$

where $\widehat\mu$ is defined as the Fourier transform 
\begin{equation}
\widehat\mu(t)=\int_{\br^d}e_t(x)\,d\mu(x).
\label{eqi6}
\end{equation}

\end{definition}

\begin{definition}

Set $N:=\#B$ and
\begin{equation}
\chi_B(t)=\frac1N\sum_{b\in B}e_b(t)
\label{eqi8}
\end{equation}
\begin{equation}
\delta_B:=\frac1N\sum_{b\in B}\delta_b\mbox{ (Dirac notation)};
\label{eqi9}
\end{equation}
then 
\begin{equation}
\chi_B(t)=\widehat\delta_B(t)
\label{eqi10}
\end{equation}
\end{definition}

\begin{lemma}\label{lem2.4}
Let $R$ and $B$ be as above, and let $(\tau_b)_{b\in B}$ be the IFS in \eqref{eqi4}. Let $\mu=\mu_{R,B}$ be the corresponding Hutchinson measure, and let $R^{T}$ be the transposed matrix. Then 
\begin{equation}
\widehat\mu(t)=\prod_{k=1}^\infty\chi_B\left((R^T)^{-k}(t)\right)
\label{eqi11}
\end{equation}
where the infinite product is absolutely convergent. 
\end{lemma}
\begin{proof}
Well known!
\end{proof}

\begin{definition}\label{def2.5}
Let a $d\times d$ matrix $R$ be given. Assume $R\in M_d(\bz)$ and that $R$ is expansive. Let $B,L\subset \br^d$ be such that $0\in B$, $0\in L$, $N=\#B=\#L$ and assume 
\begin{equation}
R^kb\cdot l \in\bz, \mbox{ for all }b\in B, l\in L, k\in\bz, k\geq0
\label{eqi12}
\end{equation}
and further that the matrix
\begin{equation}
\frac{1}{\sqrt{N}}\left(e^{2\pi iR^{-1}b\cdot l}\right)_{b\in B,l\in L}
\label{eqi13}
\end{equation}
is unitary.

Set 
\begin{equation}
\Gamma(B):=\left\{\sum_{k=0}^n R^kb_k : b_k\in B, n\in\bz_+\right\},\mbox{ and}
\label{eqi14}
\end{equation}
\begin{equation}
\Gamma(L):=\left\{\sum_{k=0}^n(R^T)^kl_k : l_k\in L, n\in\bz_+\right\}
\label{eqi15}
\end{equation}
\end{definition}

\begin{proposition}\label{pr2.5}
Let $\mu_B$ be the Hutchinson measure for $(\tau_b)_{b\in B}$ and let $\mu_L$ be the Hutchinson measure for the dual system 
\begin{equation}
\sigma_l(x):=(R^T)^{-1}(x+l),\quad(l\in L)
\label{eqi16}
\end{equation}
Then $\Gamma(B)$ is orthogonal in $L^2(\mu_L)$, and $\Gamma(L)$ is orthogonal in $L^2(\mu_B)$.
\end{proposition}

\begin{proof}
The conclusions of the proposition amount to the following: 
\begin{equation}
\ip{e_{\gamma_1}}{e_{\gamma_2}}_{\mu_B}=\widehat\mu_B(\gamma_1-\gamma_2)=\delta_{\gamma_1,\gamma_2}\mbox{ for all }\gamma_1,\gamma_2\in\Gamma(L)
\label{eq2.5.1}
\end{equation}
and
\begin{equation}
\ip{e_{\xi_1}}{e_{\xi_2}}_{\mu_L}=\widehat\mu_L(\xi_1-\xi_2)=\delta_{\xi_1,\xi_2},\mbox{ for all }\xi_1,\xi_2\in\Gamma(B)
\label{eq2.5.2}
\end{equation}

In view of \eqref{eqi12} and \eqref{eqi13} in the definition, it is easy to prove one of the stated properties and we offer the details in the verification of \eqref{eq2.5.2}. Using Lemma \ref{lem2.4}, we get 
\begin{equation}
\widehat\mu_L(\xi)=\prod_{k=1}^\infty\chi_L(R^{-k}\xi)
\label{eq2.5.3}
\end{equation}
and we will be using this for $\xi=\xi_1-\xi_2\neq 0$, with pairs of points from the set $\Gamma(B)$. So we are concerned with 
$$\xi=\sum_{k=0}^nR^k(b_k-\beta_k)$$
where $b_0,\dots,b_n,\beta_0,\dots,\beta_n\in B$. 

Substitution into \eqref{eq2.5.3} shows that the product representation for $\widehat\mu_L(\xi_1-\xi_2)$ will contain a factor of the form 
\begin{equation}
\chi_L(R^{-1}(b-\beta))
\label{eq2.5.4}
\end{equation}
where $b,\beta\in B$ distinct. 
Then 
\begin{equation}
\chi_L(R^{-1}(b-\beta))=\frac{1}{N}\sum_{l\in L}\cj{e^{2\pi iR^{-1}\beta\cdot l}}e^{2\pi i R^{-1}b\cdot l}=\ip{\mbox{row}_\beta}{\mbox{row}_b}
\label{eq2.5.5}
\end{equation}
where $\mbox{row}_\beta$ is a notation for the row with index $\beta$ in the Hadamard matrix \eqref{eqi13}, and where $\ip{\cdot}{\cdot}$ in \eqref{eq2.5.2} refers to the standard complex inner product in $\bc^N$.

Since the matrix in \eqref{eqi13} is assumed unitary, it follows that distinct rows are orthogonal. The desired conclusion follows. 
\end{proof}

\begin{corollary}\label{cor2.6}
Let the data $(R,B,L)$ be as stated in Proposition \ref{pr2.5}; then both of the orthogonal systems $\Gamma(B)$ and $\Gamma(L)$ are infinite. 
\end{corollary}

\begin{proof}
By symmetry, it is enough to prove this for $\Gamma(B)$. We will show that if $b_0,b_1,\dots$ and $\beta_0,\beta_1,\dots$ in $B$ yield the same vector
\begin{equation}
\gamma=\sum_{k=0}^nR^kb_k=\sum_{k=0}^nR^k\beta_k
\label{eq2.6.6}
\end{equation}
then $b_k=\beta_k$ for $0\leq k\leq n$. 

If not, then let $k$ be the first term in \eqref{eq2.6.6} with $b_k\neq \beta_k$. Then, from \eqref{eq2.6.6}, we get 
$$R^{-1}(b_k-\beta_k)=\sum_{i=k+1}^{n}R^{i-k-1}(\beta_i-b_i)$$
and, using \eqref{eqi12} this implies that 
$$e^{2\pi i R^{-1}b_k\cdot l}=e^{2\pi i R^{-1}\beta_k\cdot l},\mbox{ for all }l\in L.$$
But then the rows $b_k$ and $\beta_k$ in the Hadamard matrix \eqref{eqi13} coincide, so they cannot be orthogonal. The contradiction implies the corollary.

\end{proof}

\begin{corollary}\label{cor2.7}
Setting
\begin{equation}
\sigma_{\Gamma(L)}^{(B)}(t)=\sum_{\gamma\in\Gamma(L)}|\widehat\mu_B(t+\gamma)|^2,
\label{eq2.7.8}
\end{equation}
and
\begin{equation}
\sigma_{\Gamma(B)}^{(L)}(t)=\sum_{\xi\in\Gamma(B)}|\widehat\mu_L(t+\gamma)|^2,
\label{eq2.7.9}
\end{equation}
we note that the two functions are entire analytic of $t=(t_1,t_2,\dots,t_d)$, i.e., have entire analytic extensions to $\bc^d$. 

Further we have:
$$\sigma_{\Gamma(L)}^{(B)}(t)\leq 1,\mbox{ and }\sigma_{\Gamma(B)}^{(L)}(t)\leq 1\mbox{ for all }t\in\br^d.$$
\end{corollary}

\begin{proof}
While most of the conclusions follow from earlier papers, we include here the proof for $\sigma_{\Gamma(L)}^{(B)}$ in \eqref{eq2.7.8}. The analogous formula holds for $\sigma_{\Gamma(B)}^{(L)}$ in \eqref{eq2.7.9}.

Let $P_L$ denote the orthogonal projection in $L^2(\mu_B)$ onto the closed subspace $\H(L)$ spanned by $E(\Gamma(L))$, i.e., by the exponentials 
$$E(\Gamma(L)):=\{e_\gamma : \gamma\in\Gamma(L)\}\subset L^2(\mu_B).$$
Then
$$\sigma_{\Gamma(L)}^{(B)}(t)=\sum_{\gamma\in\Gamma(L)}|\widehat\mu_B(t+\gamma)|^2=\sum_{\gamma\in\Gamma(L)}|\ip{e_{-t}}{e_\gamma}_{\mu_B}|^2=\|P_Le_{-t}\|^2_{L^2(\mu_B)}$$
$$=\ip{e_{-t}}{P_Le_{-t}}_{L^2(\mu_B)}=\ip{e_0}{U(t)P_LU(-t)e_0}_{L^2(\mu_B)},$$
where, for $t=(t_1,\dots,t_d)\in\br^d$, $U(t)$ denotes multiplication by $e^{2\pi i(t_1x_1+\dots+t_dx_d)}$. Since $\supp{\mu_B}=X(B)$ is compact, $U(t)P_LU(-t)$ is entire analytic. 

The computation further shows that 
$$\sigma_{\Gamma(L)}^{(B)}(t)\leq \|U(t)P_LU(-t)\|\|e_0\|_{L^2(\mu_B)}^2=\|P_L\|=1$$
where we used the fact that $P_L$ is a projection so has norm $1$. 

\end{proof}

\begin{corollary}\label{cor2.9}
The following conclusions hold:
\begin{enumerate}
\item
$\Gamma(L)$ is total in $L^2(\mu_B)$ (i.e., it is an ONB) iff $\sigma_{\Gamma(L)}^{(B)}\equiv 1$ in $\br^d$. 
\item
$\Gamma(B)$ is total in $L^2(\mu_L)$ iff $\sigma_{\Gamma(B)}^{(L)}\equiv1$ in $\br^d$.
\end{enumerate}
\end{corollary}
\begin{proof}
See \cite{JP99}, \cite{DJ08}.
\end{proof}

\begin{theorem}\label{th2.10}
Let $R,B,L$ be as stated in Definition \ref{def2.5}. Assume further that $R^T=R$. If there is a $G\in \mbox{GL}_d(\br)$ such that $G=G^T$, 

\begin{equation}
GR=RG
\label{eqt1.1}
\end{equation}
\begin{equation}
G(B)=L
\label{eqt1.2}
\end{equation}
then the two spectral functions \eqref{eq2.7.8}, \eqref{eq2.7.9} in Corollary \ref{cor2.7} satisfy
\begin{equation}
G\Gamma(B)=\Gamma(L)
\label{eqt1.3}
\end{equation}
and
\begin{equation}
\sigma_{\Gamma(B)}^{(L)}(t)=\sigma_{\Gamma(L)}^{(B)}(Gt)\mbox{ for all }t\in\br^d.
\label{eqt1.4}
\end{equation}
\end{theorem}

\begin{proof}
It follows from \eqref{eqt1.1} and \eqref{eqt1.2} that \eqref{eqt1.3} is satisfied. 

Since $G(B)=L$, we have 
$$\chi_L(t)=\frac1N\sum_{b\in B} e_{G(b)}(t)=\frac1N\sum_{b\in B}e_b(Gt)=\chi_B(Gt)$$
and
$$\widehat\mu_L(t)=\prod_{k=1}^\infty\chi_L(R^{-k}t)=\prod_{k=1}^\infty\chi_B(GR^{-k}t)=
\prod_{k=1}^\infty \chi_B(R^{-k}Gt)=\widehat\mu_B(Gt).$$
Furthermore, for $t\in\br^d$
$$\sigma_{\Gamma(B)}^{(L)}(t)=\sum_{\gamma\in\Gamma(B)}|\widehat\mu_L(t+\gamma)|^2
=\sum_{\gamma\in\Gamma(B)}|\widehat\mu_B(G(t+\gamma))|^2=\sum_{\gamma\in\Gamma(B)}|\widehat\mu_B(Gt+G\gamma)|^2$$
$$=\sum_{\xi\in\Gamma(L)}|\widehat\mu_B(Gt+\xi)|^2=\sigma_{\Gamma(L)}^{(B)}(Gt)$$
which is the desired formula \eqref{eqt1.4}.
\end{proof}

\begin{corollary}\label{cor2.11}
If $d=1$, and $\#(B)=\#(L)=2$, then there exists $g\in\br\setminus\{0\}$ such that 
$$\sigma_{\Gamma(B)}^{(L)}(t)=\sigma_{\Gamma(L)}^{(B)}(gt)\mbox{ for all }t\in\br.$$

\end{corollary}
\section{Hadamard matrices and extreme cycles}\label{exam}

In this section we prove that iteration by some fixed scaling matrix and its transposed leads to a new Fourier duality for a pair of IFS measures. Each of the two measures arises from an affine system, but the systems are linked by a complex Hadamard matrix, see Definition \ref{def2.5} and Example \ref{ex3.1} below. Lemma \ref{lem3.3}  offers a geometric tool which allows us to test when our iterative algorithms from section \ref{comp} for the maximal orthogonal Fourier families in $L^2(\mu)$ in fact lead to orthonormal Fourier bases (ONBs) in the respective $L^2(\mu)$-spaces.

The basic building block for our examples is the matrix \eqref{eqi13}. An $N\times N$ matrix $\H$ over $\bc$ is called {\it Hadamard} if $|\H_{j,k}|=1/\sqrt{N}$ for all $j,k=1,\dots,N$. 

Complex Hadamard matrices serve as tools in combinatorics, in algebra, and in applications, see e.g. \cite{CHK97, De09, Di03, Di04, DFG+04, DA08, Pe04, XY05}. The simplest are the unitary $N\times N$ matrices which define the Fourier transform on the finite cyclic groups $\bz_N=\bz/N\bz=\{0,1,\dots,N-1\}$. If $\zeta$ is the $N$-th root of $1$, $\xi=e^{2\pi i/N}$, then 
\begin{equation}
\frac{1}{\sqrt{N}}(\zeta^{j\cdot k})_{j,k\in\bz_N}
\label{eqe1}
\end{equation}
is a complex Hadamard matrix. Moreover, the complex Hadamard matrices are closed under the following operations:
\begin{enumerate}
\item permutation of rows;
\item permutation of columns;
\item multiplication of a fixed row by a fixed phase;
\item tensor product.
\end{enumerate}

By (iv) we mean this: if $U$ is $N\times N$ and $V$ is $M\times M$ complex Hadamard matrices, then $W:=U\otimes V$ is also a $NM\times NM$ complex Hadamard matrix. To see this note that $U\otimes V$ has entries of modulus $1/\sqrt{NM}$. With the definition of $U\otimes V$ on $\bc^N\otimes\bc^M\cong\bc^{NM}$
$$(U\otimes V)(z\otimes w):=Uz\otimes Vw$$
we see that $U\otimes V$ is again unitary. 
\begin{example}\label{ex3.1}

If 
$$U=\frac{1}{\sqrt{2}}
\begin{pmatrix}
	1&1\\i&-i
\end{pmatrix}
\mbox{ and }
V=\frac{1}{\sqrt2}\begin{pmatrix}
	1&1\\1&-1
\end{pmatrix}$$
then 
$$U\otimes V=\frac12 \begin{pmatrix}
	1&1&1&1\\
	i&-i&i&-i\\
	1&1&-1&-1\\
	i&-i&-i&i
\end{pmatrix}$$
By contrast, the Fourier transform of the group $\bz_4$ is 
$$\frac12 \begin{pmatrix}
	1&1&1&1\\
	1&i&-1&-i\\
	1&-1&1&-1\\
	1&-i&-1&i
\end{pmatrix}$$

\end{example}

The main distinction between duality pairs $\mu_B$ and $\mu_L$ in the case $d=1$ and $d>1$ has to do with the following:

\begin{definition}\label{def3.2}
Let $\tau_b(x)=R^{-1}(x+b)$ be an IFS as specified in \eqref{eqi4}. If $X(B)$ is defined as in \eqref{eqi5}, a {\it finite $B$-cycle point} is a point in $X(B)$ of the form $(w,w,\dots)$ obtained by the repetition of a fixed finite word $w=(b_1\dots b_n)$. When $w$ is fixed, we denote by $x(w)$ the corresponding finite cycle point, i.e., setting 
$$x_n:=R^{-1}b_1+\dots+R^{-n}b_n$$
we get 
$$x(w)=x_n+R^{-n}x_n+R^{-2n}x_n+\dots=(I-R^{-n})^{-1}x_n=R^n(R^n-I)^{-1}x_n$$
Note that $x(w)$ can also be defined by the property 
$$\tau_{b_1}\dots\tau_{b_n}x(w)=x(w)$$
We call the set $\{x(w),\tau_{b_n}x(w),\tau_{b_{n-1}}\tau_{b_n}x(w),\dots,\tau_{b_2}\dots\tau_{b_n}x(w)\}$ the {\it $B$-cycle generated by }$x(w)$.

Similarly, one can define $L$-cycle points and $L$-cycles, associated to the IFS $(\sigma_l)_{l\in L}$ (see \eqref{eqi16}). 
Consider a system $R,B,L$ in $\br^d$ satisfying the conditions in Definition \ref{def2.5}. A finite $B$-cycle $C$ in $X(B)$ is said to be {\it $L$-extreme} if $|\chi_L(x)|=1$ for all $x\in C$. 
A finite $L$-cycle $C$ is said to be {\it $B$-extreme} if $|\chi_B(x)|=1$ for all $x\in C$.

\end{definition}

\begin{lemma}\label{lem3.3}
Let $R,B,L$ in $\br$ (so the dimension $d=1$) be as specified in Definition \ref{def2.5}. In particular, we assume $0\in B$ and $0\in L$. 
\begin{enumerate}
\item Then $\Gamma(L)$ is an ONB in $L^2(\mu_B)$ iff the only $B$-extreme cycles in $X(L)$ are the singleton $\{0\}$. 
\item Moreover, the set $\Gamma(B)$ is an ONB in $L^2(\mu_L)$ iff the only $L$-extreme cycles in $X(B)$ are the singleton $\{0\}$

\end{enumerate}
For dimensions $d>1$, the condition on the extreme cycles is only necessary, but not always sufficient for the corresponding $\Gamma$ set to be an ONB.
\end{lemma}

\begin{proof}
\cite{LW02,DJ06b}
\end{proof}

In the analysis of extreme $B$-cycles, the following lemma will be useful. See \cite[Theorem 4.1]{DJ07a} for more details. 

\begin{lemma}\label{lem3.4}
If $C$ is a $B$-extreme cycle, then for all $x\in C$, and all $b\in B$, $b\cdot x\in\bz$. 
\end{lemma}

\begin{proof}
For all $x\in C$, we must have $|\chi_B(x)|=1$. Therefore 
$$\left|\sum_{b\in  B}e^{2\pi i b\cdot x}\right|=N.$$
All the terms in the sum have absolute value 1. There are $N$ of them, and one of them is 1, since $0\in B$. Therefore we have equality in the triangle inequality and this implies that $e^{2\pi i b\cdot x}=1$ for all $b\in B$. Therefore $b\cdot x\in\bz$ for all $b\in B$.

\end{proof}

   In the following, we point out the simplifications resulting from specialization to $d = 1$. In particular (Proposition \ref{pr4.1}) we point out that the dual systems from section \ref{comp} when specialized to the particular Hadamard matrices defining finite Fourier transforms yield intriguing pairs of fractal measures in duality.

In this section, we consider a family of examples which are related to the finite cyclic group $\bz_N$, but they display fractal features and fractal duality which has not been studied earlier in duality theorems involving the finite cyclic groups.

{\bf Setting 4.1.} Fix integers $M$ and $N$ in $\bz_+$, and assume $N|M$, i.e., $M=qN$ for some $q\in\bz_+$, $q>1$. We will consider the following instance of a system $(R,B,L)$ in $\br^d$ subject to the conditions in Definition \ref{def2.5}. Set $R=M$ and $L=\{0,1,\dots,N-1\}$, and $B=qL=\{0,q,2q,\dots, (N-1)q\}$.

As a result, we see that the Hadamard matrix for $R,B,L$ is 
\begin{equation}
\frac{1}{\sqrt N}\left(\zeta^{k\cdot l}\right)_{k,l\in\bz_N},\quad \zeta=e^{2\pi i/N}.
\label{eqd1}
\end{equation}
 
 \begin{proposition}\label{pr4.1}
 Let $(R,B,L)$ be a system constructed from the numbers $N$ and $M$ as in Setting 4.1. Then 
 \begin{enumerate}
 \item $\Gamma(L)$ is an ONB in $L^2(\mu_B)$;
 \item $\Gamma(B)$ is an ONB in $L^2(\mu_L)$.
 \end{enumerate}
\end{proposition}

\begin{proof}
We will establish (i) and (ii) as an application of Lemma \ref{lem3.3}. Note that the IFS which generates $\mu_B$ is 
\begin{equation}
\tau_k^{(B)}(x)=\frac{x+kq}{M},\quad k\in\bz_N;
\label{eqd2}
\end{equation}
and for $\mu_L$ it is 
\begin{equation}
\tau_l^{(L)}(x)=\frac{x+l}{M},\quad l\in\bz_N.
\label{eqd3}
\end{equation}
For the two functions $\chi_B$ and $\chi_L$ we have 
\begin{equation}
\chi_L(t)=\frac{1}{N}\sum_{k=0}^{N-1}e_k(t)
\label{eqd4}
\end{equation}
and
\begin{equation}
\chi_B(t)=\chi_L(qt).
\label{eqd5}
\end{equation}
In exploring the extreme cycles we note that 
\begin{equation}
|\chi_L(t)|=1\mbox{ iff } t\in \bz;
\label{eqd6}
\end{equation}
and
\begin{equation}
|\chi_B(t)|=1\mbox{ iff } t\in\frac1q\bz.
\label{eqd7}
\end{equation}
Since $X(B)\subset[0,\frac{q{N-1}}{M-1}]$, $X(L)\subset[0,\frac{N-1}{M-1}]$ and $M=qN$, $q>1$, it follows that the only $L$-cycle in $X(B)$ is $\{0\}$.  And similarly, the only $B$-extreme cycle in $X(L)$ is $\{0\}$. The result now follows from Lemma \ref{lem3.3}.
\end{proof}

The next example relates to the pair of fractals in Proposition \ref{pr4.1} naturally
associated with the finite Fourier transform of $\bz_n$.
Since $\bz_n$ is its own Fourier dual, the corresponding $B$--$L$ duality may be
phrased in the language of the $\bz_n$-Fourier transform, with the two sets $B$
and $L$ essentially being a copy of $\bz_n$. Example \ref{ex4.2} below shows that if one
of the sets is changed by one point, then the ONB conclusions from Proposition \ref{pr4.1} no longer holds.

\begin{example}\label{ex4.2}

The following example shows that if the numbers in the sets $B$ or $L$ in Proposition \ref{pr4.1}
are modified, then the ONB conclusion may fail.

Let $R=8$ and $B=2\cdot\{0,1,2,3\}=\{0,2,4,6\}$ as in Proposition \ref{pr4.1}, but for $L$ now choose $L=\{0,1,2,7\}$, i.e., change 3 to 7 as compared to Proposition \ref{pr4.1} with $N=4$ and $q=2$. Now $x=1$ satisfies $\tau_7^{(L)}x=x$ and $|\chi_B(x)|=1$. Therefore $\{1\}$ is a cycle in $X(L)$ which is $B$-extreme. Combining this with the earlier observations we see that the following hold: $\Gamma(L)$ is not an ONB in $L^2(\mu_B)$, but $\Gamma(B)$ is an ONB in $L^2(\mu_L)$.
\end{example}

 The next example shows how the two spectral functions in Corollary \ref{cor2.7}
will typically be quite different when the conditions in Theorem \ref{th2.10} are
not satisfied. In the example the non-zero vectors in $B$ are linearly
independent while they are proportional in $L$. As a result there cannot be an
invertible matrix $G$ satisfying the conditions in Theorem \ref{th2.10}.
 
\begin{example}\label{ex3.2}
Let 
$$R=\begin{pmatrix}
	3&0\\
	0&3
\end{pmatrix}
,\quad
B=\left\{\begin{pmatrix}
	0\\0
\end{pmatrix}
,
\begin{pmatrix}
	1\\0
\end{pmatrix}
,
\begin{pmatrix}
	0\\1
\end{pmatrix}
\right\},\quad
L=\left\{
\begin{pmatrix}
	0\\0
\end{pmatrix}
,
\begin{pmatrix}
	1\\2
\end{pmatrix}
,
\begin{pmatrix}
	-1\\-2
\end{pmatrix}
\right\}
$$

\begin{figure}[ht]\label{fig1}
\centerline{ \vbox{\hbox{\epsfxsize 5cm\epsfbox{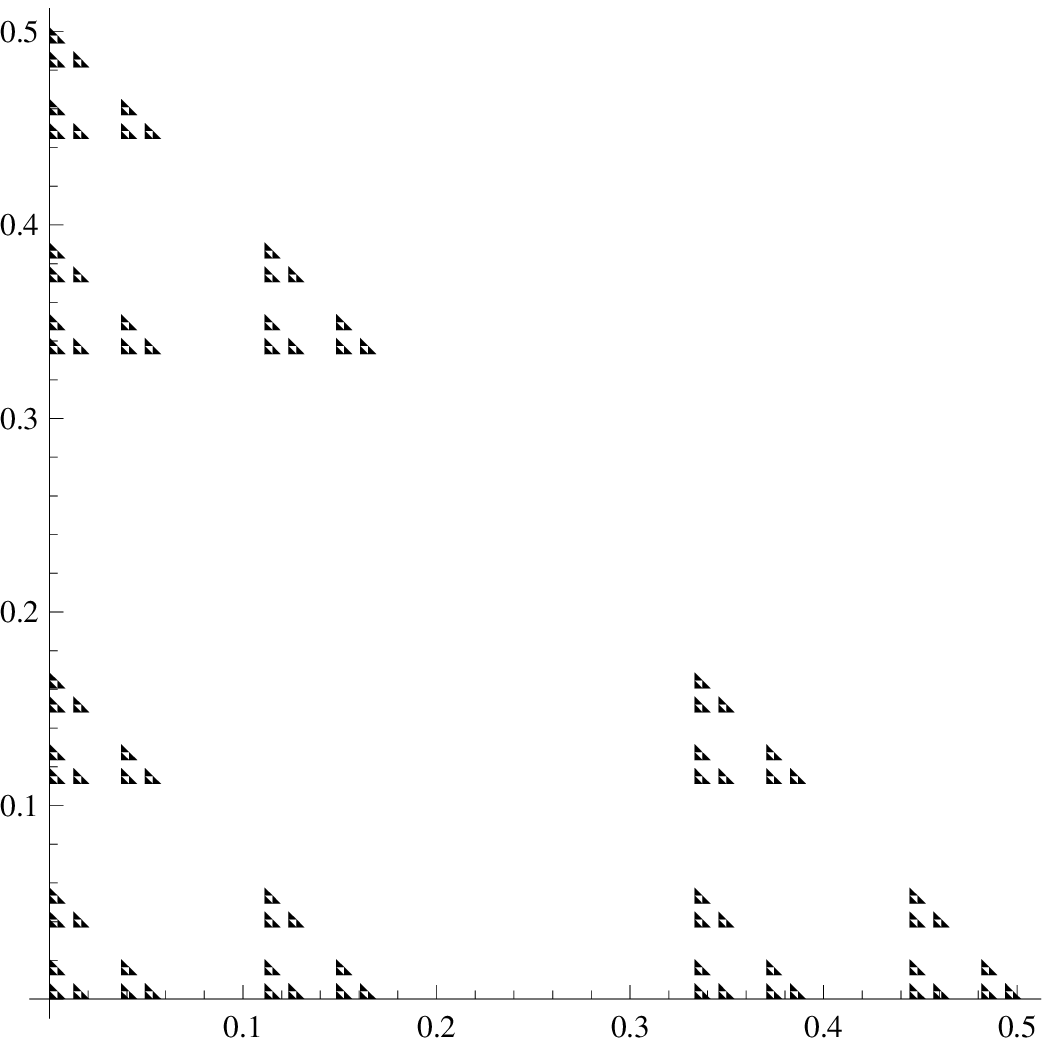}}} \hspace{3cm}
\vbox{\hbox{\epsfxsize 2.5cm\epsfbox{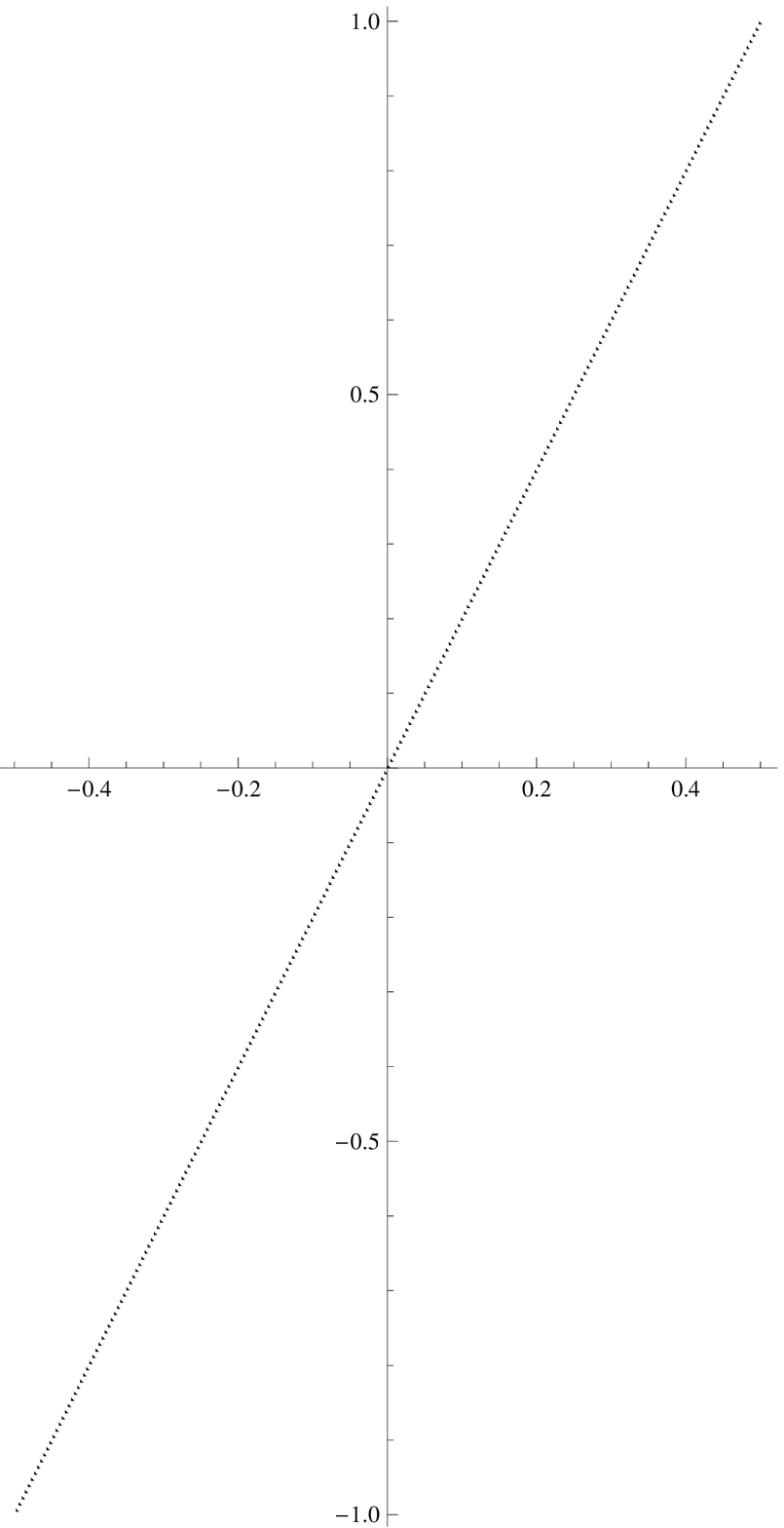}}}}
\caption{ The attractors $X(B)$ and $X(L)$}
\end{figure}

We prove that $\Gamma(L)$ is a spectrum for $\mu_B$ but $\Gamma(B)$ is not a spectrum for $\mu_L$.

Set $\zeta:=\zeta_3=e^{2\pi i/3}$. Then the matrix in \eqref{eqi13} is 
$$\frac{1}{\sqrt{3}}\begin{pmatrix}
	1&1&1\\
	1&\zeta&\zeta^2\\
	1&\zeta^2&\zeta
\end{pmatrix}$$
which is the matrix of the Fourier transform on $\bz_3$.

It is easy to see that $X(B)\subset[0,\frac12]\times[0,\frac12]$ and $X(L)\subset\left\{t \begin{pmatrix}
	1\\2
	\end{pmatrix}
:
-\frac12\leq t\leq \frac12\right\}$ contained in a line. 

Even though the dimension here is 2, and Lemma \ref{lem3.3} applies only to dimension 1, we will still be able to use it for the pair $(\mu_B,\Gamma(L))$, since the attractor $X(L)$ is contained in a line and therefore $\chi_B$ has only finitely many zeros in $X(L)$ (see \cite{DJ07b}).

Note that
$$\chi_B \begin{pmatrix}
	x\\y
\end{pmatrix}
=\frac13(1+e^{2\pi ix}+e^{2\pi i y}),\quad \chi_L \begin{pmatrix}
	x\\y
\end{pmatrix}
=\frac13(1+e^{2\pi i (x+2y)}+e^{-2\pi i(x+2y)}).$$

 Then $\left|\chi_B \begin{pmatrix}
	x\\y
\end{pmatrix}
 \right|=1$ iff $x,y\in\bz$, and $\left|\chi_L \begin{pmatrix}
	x\\y
\end{pmatrix}
\right|=1$ iff $x+2y\in\bz$.

This implies that there are no extreme $B$-cycles in $X(L)$ and so $\Gamma(L)$ is an ONB for $\mu_B$. Also, we have that $\begin{pmatrix}
	0\\1/2
\end{pmatrix}
$ is a fixed point for the map $\tau_{(0,1)^T}$ and $\left|\chi_L \begin{pmatrix}
	0\\1/2
\end{pmatrix}
\right|=1$, and therefore $\left\{ \begin{pmatrix}
	0\\1/2
\end{pmatrix}
\right\}$ is a non-trivial $L$-extreme cycle. Hence $\Gamma(B)$ is not an ONB for $\mu_L$.

\end{example}

The next proposition will show that if the set $L'$ is an integer  multiple of the set $L$ then it can only have more $B$-extreme cycles. This implies that $\Gamma(L')$ will have fewer chances of being an ONB.

\begin{proposition}\label{pr3.8}
Suppose $(R,B,L)$ and $(R,B,L')$ are as specified in Definition \ref{def2.5} and $L'=qL$ for some non-zero integer $q$. If there are some non-trivial $B$-extreme cycles in $X(L)$ then there are non-trivial $B$-extreme cycles in $X(L')$. Thus if the dimension is $1$, and if $\Gamma(L)$ is not an ONB then $\Gamma(L')$ is not an ONB either.
\end{proposition}

\begin{proof}
Let $C$ be a $B$-extreme cycle in $X(L)$. Then, by Lemma \ref{lem3.4}, $b\cdot x\in\bz$ for all $b\in B$ and $x\in C$. Then $b\cdot qx\in\bz$ for $b\in B,x\in C$. Therefore $|\chi_B(qx)|=1$ for all $x\in C$. 
Also if $x_0,x_1\in C$ and $(R^T)^{-1}(x_0+l_0)=x_1$, for some $l_0\in L$, then $(R^T)^{-1}(qx_0+ql_0)=x_1$. This shows that $qC$ is a $B$-extreme cycle in $X(L')$.

\end{proof}

 \section{Transfer operators}
 
 Recall (Corollary \ref{cor2.7}) that each of the affine fractal measures under consideration has a well defined spectral function. Here we show (Lemma \ref{lem5.1}) that the pair of spectral functions corresponding to our paired fractal measures are fixed under an associated pair of transfer operators. Since transfer operators have a rich spectral theory, this throws light on the harmonic analysis of affine fractal measures.

 Let $R,B,L$ in $\br^d$ be as specified in Definition \ref{def2.5}, and assume $0\in B$, $0\in L$. Let $(\tau_b)_{b\in B}$ and $(\tau_l)_{l\in L}$ be the dual IFSs from \eqref{eqi4} and \eqref{eqi16} respectively. 
 
 We will consider the following transfer operators
 \begin{equation}
(T_Bf)(t)=\sum_{l\in L}|\chi_B(\tau_l(t))|^2f(\tau_l(t)),
\label{eqt1}
\end{equation}
and
\begin{equation}
(T_Lf)(t)=\sum_{b\in B}|\chi_L(\tau_b(t))|^2f(\tau_b(t))
\label{eqt2}
\end{equation}

Set
\begin{equation}
E_+(T_B):=\left\{f:\br\rightarrow[0,1] : T_Bf=f, f\in C^1\mbox{ and } f(0)=1\right\}
\label{eqt3}
\end{equation}
and similarly for $E_+(T_L)$. 

\begin{lemma}\label{lem5.1}
Let $T_B$ and $T_L$ be the transfer operators in \eqref{eqt1} and \eqref{eqt2}. Let $\mathbf 1$ be the constant function one. Then
\begin{equation}
\mathbf 1\in E_+(T_B)\cap E_+(T_L),
\label{eqt4}
\end{equation}

\begin{equation}
\sigma_{\Gamma(L)}^{(B)}\in E_+(T_B),
\label{eqt5}
\end{equation}
and
\begin{equation}
\sigma_{\Gamma(B)}^{(L)}\in E_+(T_L).
\label{eqt6}
\end{equation}
Moreover, each space $E_+(T_B)$ and $E_+(T_L)$ is a convex order interval; specifically, the following implication holds: 
\begin{equation}
\mbox{If }f\in E_+(T_B)\mbox{ then } \sigma_{\Gamma(L)}^{(B)}(t)\leq f(t)\leq 1
\label{eqt7}
\end{equation}
and similarly for $E_+(T_L)$. 
\end{lemma}

\begin{proof}
Details are included in \cite{DJ06b,DJ07a,DJ07b,DJ07c}
We sketch the proof of \eqref{eqt5}. Using the definition of $\Gamma(L)$ we get 
\begin{equation}
\Gamma(L)=L+R^T\Gamma(L).
\label{eqt8}
\end{equation}
Substitution into \eqref{eq2.7.8} then yields 
$$\sigma_{\Gamma(L)}^{(B)}(t)=\sum_{l\in L}\sum_{\gamma\in\Gamma(L)}|\widehat\mu_B(t+l+R^T\gamma)|^2=
\sum_{l\in L}|\chi_B(\sigma_l(t))|^2\sum_{\gamma\in\Gamma(L)}|\widehat\mu_B((R^T)^{-1}(t+l)+\gamma)|^2$$
$$=\sum_{l\in L}|\chi_B(\tau_l(t))|^2\sigma_{\Gamma(L)}^{(B)}(\tau_l(t))=T_B(\sigma_{\Gamma(L)}^{(B)})(t)$$
as claimed.

Verification of $\sigma_{\Gamma(L)}^{(B)}(0)=1$: this is a direct computation using \eqref{eq2.7.8}, and the Hadamard axiom in \eqref{eqi13}.
See also the formula
\begin{equation}
\sigma_{\Gamma(L)}^{(B)}(t)=\ip{e_0}{U(t)P_LU(-t)e_0}_{L^2(\mu_B)}
\label{eqt9}
\end{equation}
from Corollary \ref{cor2.7}.

Indeed, setting $t=0$ into \eqref{eqt9} yields 
$$\sigma_{\Gamma(L)}^{(B)}(0)=\ip{e_0}{U(0)P_LU(0)e_0}_{L^2(\mu_B)}=\ip{e_0}{P_Le_0}_{L^2(\mu_B)}=\ip{e_0}{e_0}_{L^2(\mu_B)}=1$$
since $e_0$ is in the range of the projection $P_L$. Recall the assumption $0\in L$.

For the proof of \eqref{eqt7}, see \cite{DJ06b,DJ08}.

\end{proof}

\section{  The Cantor measure with scale-similarity factor 4}\label{canto}

 In this section we revisit a particular fractal measure $\mu$. It is a Cantor measure supported by a compact fractal contained in the real line. Its harmonic analysis was studied first in \cite{JP98} where the authors proved that it has a Fourier basis. In addition to the realization of $\mu$ as a Hutchinson measure, it is also (see \cite{JP98}) the result of a recursive Cantor construction with gap-spacing and subdivision 4.

     Here is the specific algorithm: begin with the unit interval [0, 1], subdivide by 4, and leave two gaps. Then renormalize the restriction of Lebesgue measure to two of the four subintervals that are retained. Now continue recursively. The resulting sequence of measures has a unique limit. It is the Cantor measure with fractal dimension $\frac12$.

Indeed, the limit measure $\mu$ has Hausdorff dimension $\frac12$.  And the Hausdorff dimension coincides with the scaling dimension. The authors of \cite{JP98} showed that $L^2(\mu)$ has an explicit orthonormal basis (ONB) of complex exponentials. The ONB found in \cite {JP98} is here called   $\Gamma(\{0, 1\})$.

     In the section below we explore the possibility of scaling the earlier known ONBs by integral multiples $p$. The fact that such a scaling may even lead to new ONBs is rather surprising as the scaled sets become more sparse with increasing values of the integer $p$.  Surprisingly the arithmetic properties of $p$ explain and account for when the result is again an ONB. For example we show that if $p$ is divisible by 3, then $\Gamma(\{0, p\})$ is not an ONB in $L^2(\mu)$. When $p$ does not contain the prime factor 3, we show that  $\Gamma(\{0, p\})$ may or may not an ONB. For example we prove (Proposition \ref{pr5.1}) that the case $p=5^k$ is affirmative, i.e., that $\Gamma(\{0, 5^k\})$ is an ONB in $L^2(\mu)$.

\begin{proposition}\label{pr5.1}
Let $R=4$, $B=\{0,2\}$. Then for every integer $k\geq 0$ and $L(5^k):=\{0,5^k\}$, the set $\Gamma(L(5^k))$ is an ONB for $L^2(\mu_B)$.
\end{proposition}

\begin{proof}
Using Lemma \ref{lem3.3}, we will show that for any integer $k\geq0$, there are no non-trivial $B$-extreme cycles in $X(L(5^k))$. We will do this by induction on $k$. For $k=0,1$, this is easy to check, see the conditions below.

Fix $k\geq2$. Let $x_0\in X(L(5^k))$ be a point in a $B$-extreme cycle, $x_0\neq 0$. 
We have 
$$\chi_B(t)=\frac{1}{2}(1+e^{2\pi i2\cdot t}).$$
Then $|\chi_B(x_0)|=1$ implies that $x_0\in\frac12\bz$.

We claim that any such $B$-extreme cycle point must be in $\bz$. But if $x_0$ is not in $\bz$, it must be of the form $x_0=a/2$ with $a$ odd. Let $x_1$ be the next point in the cycle, so $x_1=(x_0+l_0)/4$ for some $l_0\in\{0,5^k\}$. Since $x_1$ is a $B$-extreme cycle point we also have $x_1\in\frac12\bz$. If $l_0=5^k$ then $x_1=(a+2\cdot 5^k)/8$. Since the numerator is odd, this cannot be in $\frac12\bz$. 
If $l_0=0$ we must have $\frac12\bz\ni x_1=\frac{a}{8}$ which is again impossible. Thus $x_0\in\bz$.

Let $x_0,x_1,\dots, x_{n-1}$ be the points in this $B$-extreme cycle,  $x_0$ and let $l_0,\dots, l_{n-1}\in \{0,5^k\}$ such that 
\begin{equation}
\frac{x_0+l_0}{4}=x_1, \frac{x_1+l_1}4=x_2,\dots, \frac{x_{n-2}+l_{n-2}}4=x_{n-1}, \frac{x_{n-1}+l_{n-1}}{4}=x_0.
\label{eq5.1.0}
\end{equation}

Then 
$$x_0\equiv 4 x_1\mod 5^k, x_1\equiv 4 x_2\mod 5^k,\dots, x_{n-1}\equiv 4x_0\mod 5^k.$$
This implies that 
\begin{equation}
(4^n-1)x_i\equiv 0\mod 5^k,\quad(i\in\{0,\dots,n-1\}).
\label{eq5.1.1}
\end{equation}

Let $l$ be the largest power such that $5^l$ divides $4^n-1$. In the case when $l<k$, it follows from \eqref{eq5.1.1} that all elements of the cycle $x_i$ must be divisible by $5$. Then $x_i=5y_i$ for some $y_i\in\bz$. Dividing \eqref{eq5.1.0} by 5, it follows that $y_i$ form a $B$-extreme cycle for $L=\{0,5^{k-1}\}$, which contradicts the inductive hypothesis.

So we can assume $l=k$ so $4^n-1$ is divisible by $5^k$. We will prove that in this case $n$ is divisible by $5^{k-1}$. 
For this it is easy to prove by induction that for any $k\geq 0$:
\begin{equation}
2^{4\cdot 5^k}\equiv 1+3\cdot5^{k+1}\mod 5^{k+2}
\label{eq5.1.2}
\end{equation}
For $k=0$ this is clear. And assuming \eqref{eq5.1.2} for $k$, write $2^{4\cdot 5^k} 1+3\cdot5^{k+1}+5^{k+2}t$ then raise to the fifth power. Using the multinomial formula and keeping track of the terms not divisible by $5^{k+3}$, it follows that \eqref{eq5.1.2} is true for $k+1$ as well.

Now consider the multiplicative group $(\bz_{5^k})^*$ of elements in $\bz_{5^k}$ that are relatively prime to $5^k$. 
Equation \eqref{eq5.1.2} for $k-1$, shows that the order of $2$ in this group divides $4\cdot 5^{k-1}$, and \eqref{eq5.1.2} for $k-2$ shows that the order of $2$ in this group cannot divide $4\cdot5^{k-2}$. Therefore the order of $2$ is divisible by $5^{k-1}$. 

Then, if $2^{2n}\equiv 1\mod 5^k$ this implies that the order of 2 must divide $2n$ so $5^{k-1}$ divides $n$.

But if $5^{k-1}$ divides $n$, then $n\geq 5^{k-1}$. On the other hand $n$ is the length of the cycle. This means that we have at least $5^{k-1}$ points in $\bz\cap X(L)$. 

However, $X(L)$ is contained in the interval $[0, \sum_{i=0}^\infty 5^k/4^i]=[0,5^k/3]$. There are at most $5^k/3$ non-zero integers in this interval. But, as above, for all points in the cycle we must have $x_i+l_i$ divisible by $4$, so $x_i\equiv 0\mod 4$ or $x_i\equiv -5^k\mod 4$. These are only 2 equivalence classes mod 4, so the number of integers  in $X(L)$ that can be on an extreme cycle is at most $(5^k/3)/2<5^{k-1}$. But since the length of the cycle , $n$, is at least $5^{k-1}$, we reach a contradiction. 

 Thus, there are no non-trivial $B$-extreme cycles, and therefore $\Gamma(\{0,5^k\})$ is an ONB. 

\end{proof}
\begin{remark}
Proposition \ref{pr5.1} shows that an ONB for a fractal measure can have the corresponding fractional upper Beurling density arbitrarily small.

The fractional upper Beurling density is defined as follow (see \cite{CKS08}): for a discrete subset $\Lambda$ of $\br^d$, and for $\alpha>0$, the $\alpha$-upper Beurling density of $\Lambda$ is defined by 
$$\mathcal D_\alpha^+(\Lambda)=\limsup_{h\rightarrow\infty}\sup_{x\in\br^d}\frac{\#(\Lambda\cap(x+h[-1,1]^d))}{h^\alpha}$$
and the upper Beurling dimension of $\Lambda$ is defined by 
$$\dim^+(\Lambda)=\sup\{\alpha>0 : \mathcal D_\alpha(\Lambda)>0\}.$$

It was proved in \cite{DHSW09} that, under some mild assumptions, for any set $\Lambda$ such that the exponentials $E(\Lambda):=\{e_\lambda : \lambda\in\Lambda\}$ form a Bessel sequence in $L^2(\mu)$ (in particular ONBs), the upper Beurling dimension is equal to the Hausdorff dimension which in this case is $\ln 2/\ln 4=1/2$, and the $1/2$-upper Beurling density $\mathcal D_{1/2}^+(\Lambda)$ is finite. 

Thus we have $\mathcal D_{1/2}^+(\Gamma_1)<\infty$, where $\Gamma_1:=\Gamma(\{0,1\})$. We check also that $\mathcal D_{1/2}^+(\Gamma_1)>0$.

Take $x=0$ and $h=\sum_{k=0}^{n-1}4^k=(4^n-1)/3$ in the definition of the Beurling density. Then 
$\# (\Gamma_1\cap[-h,h])=2^n$ (since we have two digits in $\{0,1\}$ and $n$ positions to write the elements in $\Gamma_1$). Then 
$$\mathcal D_{1/2}^+(\Gamma_1)\geq \limsup_{n\rightarrow\infty}\frac{2^n}{\left(\frac{4^n-1}{3}\right)^{1/2}}=\sqrt{3}>0.$$

Thus $$0<\mathcal D_{1/2}^+(\Gamma_1)<\infty.$$
Also, it is easy to see that 
$$\mathcal D_{1/2}^+(5^k\Gamma_1)=\frac{1}{(5^k)^{1/2}}\mathcal D_{1/2}^+(\Gamma_1),$$
and therefore 
$$\lim_{k\rightarrow\infty}\mathcal D_{1/2}^+(5^k\Gamma_1)=0$$

\end{remark}

Let $d=1$, $R=4$ and $B=\{0,2\}$ so the measure $\mu_B$ satisfies 
\begin{equation}
\widehat\mu_B(t)=e^{i\frac{2\pi t}{3}}\prod_{k=1}^\infty\cos\left(\frac{2\pi t}{4^k}\right)
\label{eqp.1}
\end{equation}
Set 
\begin{equation}
\Gamma:=\Gamma_1:=\left\{\sum_{k=0}^na_k4^k : a_k\in\{0,1\}, n\in\bz_+\right\}
\label{eqp.2}
\end{equation}
i.e., $\Gamma=\Gamma(L_1)$ with $L_1=\{0,1\}$.

\begin{proposition}\label{pr5.3}
Let $p\in\bz_+$, $p> 1$, and $L_p:=\{0,p\}$. Then the following conditions are equivalent:
\begin{enumerate}
\item $\Gamma(L_p)=p\Gamma_1$ is orthogonal (not necessarily ONB) in $L^2(\mu_B)$;
\item $\Gamma(B)=2\Gamma_1$ is orthogonal in $L^2(\mu_{L_p})$.
\item $p$ is odd.  

\end{enumerate} 

\begin{proof}
See Proposition \ref{pr2.5}. 
\end{proof}

\end{proposition}

\begin{corollary}\label{cor5.4}
Let $B$ and $L_p=\{0,p\}$ be as in Proposition \ref{pr5.3}, i.e., assume $p$ is odd. Then the following conditions are equivalent:
\begin{enumerate}
	\item The only $B$-extreme cycles in $X(L_p)$ are the trivial singleton $\{0\}$.
	\item The only $L_p$-extreme cycles in $X(B)$ are the trivial singleton $\{0\}$.
	\item The orthogonal sets in Proposition \ref{pr5.3} are ONBs. 
\end{enumerate}

\end{corollary}

\begin{proof}
Combine Lemma \ref{lem3.3} and Theorem \ref{th2.10}.

\end{proof}

Here is a list of non-trivial $B$-extreme cycles for all the values of $p\leq 100$. This consists of all odd multiples of 3, and the only such $p$ not divisible by 3 is 85.  For the odd values of $p$ that do not appear in the table, there are no such cycles so $\Gamma(L_p)$ is an ONB. Therefore, the list of positive integers $p$ less than 100  such that $\Gamma(\{0, p\})$ is an ONB in $L^2(\mu_B)$ is:  1, 5, 7, 11, 13, 17, 19, 23, 25, 29, 31, 35, 37, 41, 43, 47, 49, 53, 55, 59, 61, 65, 67, 71, 73, 77, 79, 83, 89, 91, 95, 97. 

$$\begin{tabular}{|c|l|}
\hline
$p$&\mbox{ cycles}\\\hline
3&\{1\}\\\hline
9&\{3\}\\\hline
15&\{4,1\},\{5\}\\\hline
21&\{7\}\\\hline
27&\{9\}\\\hline
33&\{11\}\\\hline
39&\{13\}\\\hline
45&\{12,3\},\{15\}\\\hline
51&\{13,16,4,1\},\{17\}\\\hline
57&\{19\},\\\hline
63&\{16,4,1\},\{17,20,5\},\{21\}\\\hline
69&\{23\}\\\hline
75&\{20,5\},\{25\}\\\hline
81&\{27\}\\\hline
85&\{23,27,28,7\}\\\hline
87&\{29\}\\\hline
93&\{31\}\\\hline
99&\{33\}\\\hline
\end{tabular}
$$

\begin{remark}\label{rem5.5}
  It follows from the analysis of the extreme cycles introduced in Proposition \ref{pr5.3} and \ref{cor5.4} that when $p$ is divisible by 3, then there are non-trivial $B$-extreme cycles in $X(\{0, p\})$. And as a result, the orthogonal sets from Proposition \ref{pr5.3} cannot be total in the respective $L^2(\mu)$-Hilbert spaces. However this implication only goes one way as is illustrated in the table for the case of $p=85$ : Even when $p$ does not have 3 as a prime factor there may indeed be non-trivial $B$-extreme cycles in $X(\{0, p\})$. For $p = 85$, there is a $B$-extreme cycle in $X(\{0, 85\})$ of length four.

\end{remark}

\section{Finite cycles}\label{fini}

  In this section we explore variations in the lists of $B$-extreme cycles for our family of 1D-examples, algebraic and geometric properties; and we identify two classes of such cycles. Our conclusions have direct relevance to orthogonal harmonic analysis of fractal measures as we developed it in the earlier sections of this paper.  But these finite cycles are of independent interest as they have wider significance for factorizations of Mersenne numbers (among other topics); a subject with a multitude of applications: combinatorics, number theory, and encryption; see for example \cite{FFS09, Mu94, MP04, Od78, Vas06}.

    In addition, we mention that finite cycles that may be distributed on particular lattices have uses in other problems in analysis, for example in the study of representations of the Cuntz $C^*$-algebras. In \cite{BrJo99}, the authors introduced a family of permutative representations of each of the Cuntz algebras $\mathcal O_n$. They showed that finite affine cycles distributed on certain associated lattices account for the orthogonal decompositions of these representations. A representation of $\mathcal O_n$ acting in a Hilbert space $\mathcal H$ is said to be permutative if it permutes the vectors in some orthonormal basis for $\mathcal H$. The papers \cite{DJ07b, DJ07c} offer yet other applications of finite cycles with lattice coordinates. For related questions about representations of $\mathcal O_n$, see also \cite{SZ08}.

       Note that in the study of fractals and dynamics, there are multitudes of families of finite cycles; but our present restriction that the particular cycles be $B$-extreme cuts down the number of cases. And for the discussion below, the $B$-extremality of a particular cycle $\mathcal C$ turns out to be equivalent to  requiring $\mathcal C$ to be contained in a certain lattice.

We consider a Hadamard system in Definition \ref{def2.5} of the following form in one dimension, i.e., $d=1$.

{\bf Setting.} Let $R=2n$, $B=\{0,2\}$, $p\in\bz_+$ odd, and set $L=L_{n,p}=\{0, np/2\}$. 
\medskip

Then the conditions in Definition \ref{def2.5} are satisfied and the Hadamard matrix is 
$$\frac{1}{\sqrt2} \begin{pmatrix}
	1&1\\1&-1
\end{pmatrix}
$$

In this section we compute some special $\tau_L$-cycles which are $B$-extreme. This is of interest since the ONB condition in Lemma \ref{lem3.3} holds iff the only $B$-extreme $\tau_L$-cycle is the singleton $\{0\}$.

Recall from Lemma \ref{lem3.4} that a finite $\tau_L$-cycle $\mathcal C$ is $B$-extreme iff $\mathcal C\subset\frac12\bz$, see also Definition \ref{def3.2}.

Set 
\begin{equation}
\tau_0(t)=\frac{t}{2n}\mbox{ and } \tau_1(t)=\frac{t+\frac{np}2}{2n}=\frac{t}{2n}+\frac{p}4
\label{eq6.1}
\end{equation}
and let $\omega=(\omega_1\omega_2\dots\omega_l)$, $\omega_i\in\{0,1\}$ be a finite word. 
Set
\begin{equation}
\tau_\omega=\tau_{\omega_1}\circ\tau_{\omega_2}\circ\dots\circ\tau_{\omega_l}.
\label{eq6.2}
\end{equation}

\begin{lemma}\label{lem6.1}
If $p$ is divisible by $2n-1$ then there are non-trivial $B$-extreme $\tau_{L_{n,p}}$-cycles of length one.
\end{lemma}

\begin{proof}
Let $p=m(2n-1)$, $m\in\bz$. Then the solution $t$ to $\tau_1(t)=t$ is 
\begin{equation}
t=\frac{np}{2(2n-1)}=\frac{mn}{2}\in\frac12\bz.
\label{eq6.4}
\end{equation}
Hence $\left\{\frac{mn}2\right\}$ is a $B$-extreme $\tau_{L_{n,p}}$-cycle.
\end{proof}

\begin{corollary}\label{cor6.2}
(i)
For $R=4$, $B=\{0,2\}$, and $L=\{0,3\}$, the singleton $\mathcal C=\{1\}$ is a $B$-extreme $\tau_L$-cycle; and so in particular $\Gamma(\{0,3\})$ is not an ONB in $L^2(\mu_{1/4})$.

(ii) For $R=6$, $B=\{0,2\}$ and $L=\left\{0,\frac{15}2\right\}$, the singleton $\mathcal C=\{\frac32\}$ is a $B$-extreme $\tau_L$-cycle.

(iii) For $R=8$, $B=\{0,2\}$, and $L=\{0,14\}$, the singleton $\mathcal C=\{2\}$ is a $B$-extreme $\tau_L$-cycle.
\end{corollary}

\begin{definition}\label{def6.3}
Let $n\in\bz_+$ be given, and set $R=2n$, $B=\{0,2\}$ and let $\mu=\mu_{1/2n}=\mu_B$ be the fractal measure from section \ref{comp}. We say that $p\in\bz_+$ odd is {\it admissible} iff there are non-trivial $B$-extreme $\tau_L$-cycles. Here $L=L_{n,p}=\{0,np/2\}$.

Note that if $n=2$ then $p=3$ is admissible. A consequence of Lemma \ref{lem6.1} is that for $n\in\bz_+$, $p=2n-1$ is admissible, and, with Proposition \ref{pr3.8}, we obtain that any $p$ divisible by $2n-1$ is admissible.

\end{definition}

{\bf Question.} Are there admissible values of $p$ not divisible by $2n-1$?

The next theorem offers an affirmative answer.

\begin{theorem}\label{th6.5}
Let $n\in\bz_+$ be given. Set 
\begin{equation}
p=\sum_{i=0}^{2n-1}(2n)^i.
\label{eq6.5}
\end{equation}
Then $p$ is admissible and not divisible by $2n-1$.  There are associated $B$-extreme cycles of length $2n$. 
\end{theorem}

\begin{proof}
Let $p$ be given by \eqref{eq6.5}. Then $p\equiv 1\mod 2n-1$, so it  is not divisible by $2n-1$. We shall need 
\begin{equation}
p^*:=\sum_{i=0}^{2n-2}(2n)^i.
\label{eq6.6}
\end{equation}
Note that $p^*\equiv 0\mod 2n-1$. So $p^*=m(2n-1)$ for some $m\in\bz$. 

Consider the solution $t$ of the fixed point equation 
$$\tau_0\tau_1^{2n-1}(t)=t$$
It is easy to see that 
$$t=\frac{\frac{np}{2} \sum_{i=0}^{2n-2}(2n)^i}{(2n)^{2n}-1}=\frac{np p^*}{2(2n-1)p}=\frac{np^*}{2(2n-1)}=\frac{nm}{2}\in\frac12\bz.$$

We check by induction that $\tau_1^k(t)\in\frac12\bz$ for $k\in\{0,1,\dots, 2n-1\}$.
We use that $p=p^*+(2n)^{2n-1}$. We have
$$\tau_1(t)=\frac{t+\frac{np}{2}}{2n}=\frac{\frac{mn}{2}+\frac{np}{2}}{2n}=\frac{n}{2}\frac{\frac{p^*}{2n-1}+p^*+(2n)^{2n-1}}{2n}=
\frac{n}{2}\left(\frac{(2n)p^*}{2n(2n-1)}+(2n)^{2n-2}\right).$$
So 
$$\tau_1(t)=\frac{mn}{2}+\frac{n}{2}(2n)^{2n-2}=t+\frac{n}{2}(2n)^{2n-2}.$$
Then 
$$\tau_1^2(t)=\tau_1\left(t+\frac{n}{2}(2n)^{2n-2}\right)=\tau_1(t)+\frac{1}{2n}\frac{n}{2}(2n)^{2n-2}=t+\frac{n}{2}(2n)^{2n-2}+\frac{n}{2}(2n)^{2n-3},$$
where we used the previous step in the last equality. 

By induction we get 

$$\tau_1^k(t)=t+\frac{n}{2}(2n)^{2n-2}+\frac{n}{2}(2n)^{2n-3}+\dots+\frac{n}{2}(2n)^{2n-k-1}$$
for all $k\leq 2n-1$. And this shows that $\tau_1^k(t)$ is in $\frac12\bz$. Since $\tau_0\tau_1^{2n-1}(t)=t$, this implies that the entire cycle is in $\frac12\bz$, so it is a $B$-extreme $\tau_{L_{n,p}}$-cycle.

\end{proof}

\begin{remark}\label{rem5.9}
An application of Theorem \ref{th6.5} to $R=4$ and $p=1+4+4^2+4^3=85$ accounts for the cycle $\mathcal C=\{23,27,28,7\}$ from the table before Remark \ref{rem5.5}.
It is the smallest admissible $p$ not divisible by 3.

For $R=6$ and $p=1+6+6^2+6^3+6^4+6^5=9331$ on has $L=\{0,27993/2\}$ the $B$-extreme cycle of length 6 is 
$$\mathcal C=\left\{ \frac{4821}2,\frac{5469}2, \frac{5577}2, \frac{5595}2, 2799,\frac{933}2\right\}$$ 

 It is also the smallest admissible $p$ not divisible by $5$. 

For $R=8$ and $p=\sum_{i=0}^7 8^i=2396745$ one has $L=\{0,4793490\}$ and the $B$-extreme cycle of length 8 is 
$$\mathcal C=\left\{ 609886, 675422, 683614, 684638, 684766, 684782, 684784, 85598\right\}.$$

We used the following Mathematica program to check for extreme cycles. 

\begin{verbatim}
r = 8;
NextCycle[x_, p_] := 
  If[IntegerQ[2*x/r], x/r, If[IntegerQ[2*(x + p)/r], (x + p)/r, -1]];
pmax = r*(r^r - 1)/(4*(r - 1))
For[p = pmax - r/2, p < pmax, p = p + r/2; 
 If[! IntegerQ[2*p/(r - 1)],
  For[ix = 1/2, ix <= IntegerPart[2*p/(r - 1)]/2, ix = ix + 1/2, 
   x = NextCycle[ix, p]; flag = 0;
   While[x != -1 && x != ix, x = NextCycle[x, p]];
   If[x != -1, cycle = {}; If[flag == 0, Print["p= ", p]; flag = 1]; 
    Print["Cycle pt "]; i = 1;
    While[i == 1, x = NextCycle[x, p]; Print[" ", x]; 
     If[x == ix, i = 0]]; Print["End Cycle "]
    ]
   ]
  ]
 ]



\end{verbatim}

\end{remark}

\begin{acknowledgements}
The co-authors are grateful for helpful conversations with members of the analysis groups at their respective universities. In addition we are grateful to Professor Michael Reid (at UCF) who helped us with the proof of Proposition \ref{pr5.1}.
\end{acknowledgements}

\end{document}